\documentclass[12pt]{amsart}
\usepackage{calc,pifont}
\usepackage{graphicx,adjustbox}
\usepackage{mathrsfs}
\usepackage{amssymb,amsbsy}
\usepackage{amscd}
\usepackage{enumerate}
\usepackage{setspace}
\usepackage{tikz}
\usepackage{bm}
\usepackage{cite}
\usepackage{graphicx}
\usepackage{hyperref}
\allowdisplaybreaks

\usepackage{kpfonts}
\usepackage{stackengine}
\usepackage{calc}
\newlength\shlength
\newcommand\xshlongvec[2][0]{\setlength\shlength{#1pt}%
  \stackengine{-5.6pt}{$#2$}{\smash{$\kern\shlength%
    \stackengine{7.55pt}{$\mathchar"017E$}%
      {\rule{\widthof{$#2$}}{.57pt}\kern.4pt}{O}{r}{F}{F}{L}\kern-\shlength$}}%
      {O}{c}{F}{T}{S}}
  
\usepackage[letterpaper]{geometry}
\theoremstyle{definition}
\newtheorem{theorem}{Theorem}[section]
\newtheorem{lemma}[theorem]{Lemma}
\newtheorem{proposition}[theorem]{Proposition}
\newtheorem{definition}[theorem]{Definition}
\newtheorem{corollary}[theorem]{Corollary}
\newtheorem{conjecture}[theorem]{Conjecture}

\newcommand{\inpr}[1]{\left\langle #1 \right\rangle}	
\newcommand{\ovl}[1]{\overline{#1}}	
\newcommand{\parn}[1]{\left( #1 \right)}
\newcommand{\conv}[1]{\mathrm{conv}\parn{#1}}	

\newcommand{\Z}{{\mathbb{Z}}}

\newcommand{\C}{{\mathbb{C}}}
\newcommand{\N}{{\mathbb{N}}}
\newcommand{\Hil}{\mathcal H}
\newcommand{\calM}{\mathcal M}
\newcommand{\Ci}{\mathcal C}

\renewcommand{\Re}{\ensuremath{\mathrm{Re}}}
\renewcommand{\Im}{\ensuremath{\mathrm{Im}}}

\author{\sc{Benjam\'in A.~Itz\'a-Ortiz}} 
\author{Rub\'en A.~Mart\'inez-Avenda\~no}

\address{Centro de Investigaci\'on en Matem\'aticas, Universidad Aut\'onoma del Estado de Hidalgo, Pachuca, Hidalgo, Mexico}
\address{Departamento Acad\'emico de Matem\'aticas, Instituto Tecnol\'ogico Aut\'onomo de M\'exico, Mexico City, Mexico}

\thanks{The second author's research is partially supported by the Asociaci\'on Mexicana de Cultura A.C.}

\begin{document}

\title{The numerical range of a class of periodic tridiagonal operators}

\begin{abstract} 
 In this paper we compute the closure of the numerical range of certain periodic tridiagonal operators. This is achieved by  showing that the closure of the numerical range of such operators can be expressed as the  closure of the convex hull of the uncountable union of numerical ranges of certain symbol matrices.  For a special case, this result can be improved so that 
 it is the convex hull of the union of the numerical ranges of only two matrices.  A conjecture is stated for the general case.
\end{abstract}
\maketitle

\section*{Introduction}

Given  $b=(b_i)_{i\in\Z}$ a biinfinite sequence in the total shift space $\mathcal A^\Z$, where $\mathcal A$ is a finite set of complex numbers, we associate a tridiagonal operator $A_b\colon \ell^2(\Z)\to\ell^2(\Z)$   defined as
 \[
 A_{b}= 
\left( \begin{array}{ccccccc}
\ddots & \ddots & & & & & \\
\ddots & 0 & 1 & & & & \\
& b_{-2} & 0 & 1 & & & \\
& & b_{-1} & \framebox[0.4cm][l]{0} & 1 & & \\
& & & b_{0} & 0 & 1 & \\
& & & & b_{1} & 0 & \ddots \\
& & & & & \ddots & \ddots
\end{array} \right) 
\]
where the rectangle marks the matrix entry at $(0,0)$. When $\mathcal A$ is the set $\{-1,1\}$,  the corresponding operator $A_b$ is related to the so called ``hopping sign model'' introduced in \cite{Feinberg} and subsequently studied in many other works, such as \cite{BebianoEtAl,CWChLi10,CWChLi13,ChD,Hagger,HaggerJFA,HI-O2016}, just to name a few. 

Except for particular cases, there are not general results to establish neither the spectrum nor  the numerical range of $A_b$.
Recall that the numerical range of a bounded operator $T$ on a Hilbert space $\mathcal H$ is defined as the set
\[
W(T):=\left\{ \langle T x, x \rangle \, : \, x \in {\mathcal H}, \, \| x \|=1 \right\}.
\]
This set turns out to have many nice properties and gives a lot of information about the operator. We mention here some of the properties of the numerical range that we will use in the sequel (most of the proofs are easy and they can be found in, for example, \cite{HoJo,GusRao}). For a bounded operator $T: \Hil \to \Hil$ we have
\begin{itemize}
    \item $W(T)$ is a bounded convex set.
    \item If $\Hil$ is finite dimensional, then $W(T)$ is a closed set.
    \item For every $a, b \in \C$, we have $W(a T+b)=a W(T)+b$.
\item $\Re(W(T))=W(\Re(T))$ where $\Re(T)=\frac{1}{2}(T+T^*)$.
\item If $\Hil$ is finite dimensional and $T$ is a Hermitian matrix, then $W(T)=[\lambda^-,\lambda^+]$, where $\lambda^-$ is the smallest eigenvalue of $T$ and $\lambda^+$ is the largest eigenvalue of $T$.
\item If $\Hil=\C^2$ and $T=\begin{pmatrix}0 & a \\ b & 0 \end{pmatrix}$, then $W(T)$ is the ellipse with focii $\pm \sqrt{a b}$ and major axis of length $|a|+|b|$.
\end{itemize}

In this paper we advocate to investigate the numerical range of $A_b$ when $b$ is a $n$-periodic sequence.  Following work of Bebiano et al. \cite{BebianoEtAl}, we find that $\overline{W(A_b)}$ can be expressed as the closure of the convex hull of uncountable union of numerical ranges of certain  symbol matrices in  $M_n(\C)$. For the case $\mathcal A=\{0,1\}$ and $b$ is a $2$-periodic sequence with period word $01$, we explicitly determine $\overline{W(A_b)}$ as the convex hull of the union of  numerical ranges of only two matrices in $M_2(\C)$. We then state a conjecture where we claim $\overline{W(A_b)}$ to be the convex hull of numerical ranges of two matrices in $M_n(\C)$ when $b$ is $n$-periodic with period word $0\cdots01$.  

We divide this work in three sections. In Section~1, we introduce the necessary concepts and notions used in the rest of the paper. In particular we justify that we may restrict to operators on $\ell^2(\N_0)$ rather than $\ell^2(\Z)$. In Section~2 we state and prove the main results of the paper. Finally, in Section~3, we state a conjecture which would greatly improve the computation the numerical range of our tridiagonal operators and verify it for the case $n=2$.

The authors gratefully acknowledge the referee's comments and suggestions which helped to improve the readability of this paper. Furthermore, the authors are deeply indebted to the referee for proposing a proof of our conjecture stated in Section~3, which we expect to be published at a later date.

\section{Preliminaries}
In this section we introduce the necessary notation and terminology needed in the paper. Since one-sided infinite tridiagonal operators are far more used than their biinfinite counterparts, we introduce notation for one-sided infinite periodic tridiagonal operators and work with them for the rest of the paper. At the end of this section, we prove that the closure of the numerical range of one-sided periodic tridiagonal operators coincides with the closure of the numerical range of their biinfinite counterparts. 
 
Fix an alphabet $\mathcal A$, that is to say,  a finite subset of complex numbers. 
For $m\in\Z$, denote the set  $\Z_{\geq m}=\{t \in \Z \, : \, t \geq m\}$ so in particular we declare $\N_0=\Z_{\geq 0}$. A sequence $a$ in $\mathcal A^\Z$ (or in $\mathcal A^{\N_0}$) is said to be $n$-periodic if $n$ is a positive integer such that $a_k=a_{k+n}$ for all  $k\in \Z$ (respectively, for all $k \in \N_0$). Therefore, if $a\in \mathcal A^\Z$ (or in $\mathcal A^{\N_0}$) is $n$-periodic then we refer to the finite subsequence $a_0 a_1\cdots a_{n-1}$  as the period word of $a$.  

Recall that, given a bounded operator $A$ on a Hilbert space $\Hil$, and given a subspace $\calM$ of $\Hil$, the {\em compression} of $A$ to $\calM$ is defined as the operator $P A \rvert_{\calM}$, where $P$ is the orthogonal projection onto $\calM$.

For a given $n\in\N$, let $a$, $b$ and $c$ be $(n+1)$-periodic infinite sequences in $\mathcal A^{\N_0}$. We will denote by $T=T(a,b,c)$ the  $(n+1)$-periodic tridiagonal operator on $\ell^2(\N_0)$ given by
\[
T= 
\left( \begin{array}{ccccccccccc}
b_0 & c_0 & & & & & & & &\\
a_1 & b_1 & c_1 & & & & & & & \\
    & a_{2} & b_2 & c_2 & & & & & & \\
    &       & \ddots & \ddots  & \ddots & & & & & \\
    & & & a_{n} & b_n & c_n & & & &\\
    & & & & a_{0} & b_0 & c_0 & & &\\
    & & & & & \ddots & \ddots & \ddots & &\\
    & & & & & &a_{n-1} & b_{n-1}& c_{n-1} & \\
    & & & & & & &  a_n& b_n  & c_n &\\
    & & & & & & & & \ddots &  \ddots&\ddots
\end{array} \right). \]

We should observe that $T$ is a bounded operator since the sum of the moduli of the entries in each column (and in each row) is uniformly bounded (see, e.g., \cite[Example 2.3]{Kato}). The same is true for the biinfinite matrix $A_b$, as long as the biinfinite sequence arises from a finite alphabet.

Fix $s \in \N$, $s>1$, and let $m:=s(n+1)$. We then define  $C_m$ to be a corresponding circulant matrix of $T$, an $m\times m$ matrix, as: 
\[
C_{m}= 
\left( \begin{array}{cccccccccc}
b_0 & c_0 & 0 & & & & & & 0 & a_0\\
a_1 & b_1 & c_1 & 0 & & & & &  & 0 \\
0   & a_{2} & b_2 & c_2 & 0 & & & & & \\
    & \ddots & \ddots & \ddots  & \ddots & \ddots & & & & \\
    & & 0 & a_{n} & b_n & c_n & 0 & & &\\
    & & & 0 & a_{0} & b_0 & c_0 & 0 & &\\
    & & & & \ddots & \ddots & \ddots & \ddots &\ddots \\
    & & & & & \ddots & \ddots & \ddots & \ddots  & \\
0 & & & & & & 0 &a_{n-1} & b_{n-1}& c_{n-1}  \\
c_n & 0 & & & & & &  0 & a_n& b_n 
\end{array} \right). 
\]
We observe that by removing the last column and the last row of $C_m$ we obtain a matrix which is a compression of $T$. This observation will be useful later.

Finally, if $n>1$, for each $\phi \in [0, 2\pi)$, we define the corresponding symbol of $T$,  as the following  $(n+1)\times (n+1)$ matrix
\[
T_{\phi}= 
\left( \begin{array}{ccccccc}
 b_0 & c_0 & 0 & & & 0 & a_0 e^{-i\phi}\\
a_1 & b_1 & c_1 & 0 & & & 0\\
0   & a_{2} & b_2 & c_2 & 0 & & \\
    & \ddots & \ddots &\ddots & \ddots & \ddots   & \\
& & 0 & a_{n-2} & b_{n-2} & c_{n-2} & 0\\
0 & & & 0 & a_{n-1} & b_{n-1} & c_{n-1} \\
c_n e^{i\phi} & 0 & & & 0 & a_n & b_n
\end{array} \right); 
\]
while the symbol  of $T$ for $n=1$ is the $2\times 2$ matrix 
\begin{equation}\label{eq:symbol2}
    T_{\phi}=
        \begin{pmatrix}
        b_0 & c_0+a_0e^{-i\phi}\\ 
        a_1+c_1 e^{i\phi} & b_1  
        \end{pmatrix}.
\end{equation}

To conclude this section, we will justify why the (closure of the) numerical ranges of our tridiagonal operators on $\ell^2(\Z)$ and their compressions to $\ell^2(\N_0)$ are equal.

\begin{proposition}
     Let $b=(b_i)_{i\in\Z}$ be an $n$-periodic biinfinite sequence and let $A_b$ be the corresponding biinfinite tridiagonal operator. Let $P$ be  the projection of $\ell^2(\Z)$ onto the subspace $\ell^2(\N_0)$. If $T$ is the compression of $A_b$ to $\ell^2(\N_0)$, i.e.
    $T=P A_b \rvert_{\ell^2(\N_0)}$, then
     \[
      W(T) \subseteq W(A_b) \subseteq \overline{W(T)}.
   \]
   In particular,
   \[
     \overline{W(T)}=\overline{W(A_b)}.
     \]
\end{proposition}   
\begin{proof}
    The inclusion $W(T) \subseteq W(A_b)$ is straightforward. 
    
    For the other inclusion, we need the following definitions. For each $m\in\Z$ let $P_m$ be the projection of $\ell^2(\Z)$ onto the subspace $\ell^2(\Z_{\geq m})$. Define $T_m : \ell^2(\Z_{\geq m}) \to \ell^2(\Z_{\geq m})$ by $T_m =P_m A_b \rvert_{\ell^2(\Z_{\geq m})}$. Observe that $T_0=T$ and $P_0=P$. It is clear that, for each $k\in \Z$, the operators $T_{kn}$ and $T$ are unitarily equivalent (just write out the matrices for each operator) and hence $W(T_{kn})=W(T)$.

    Now, let $\lambda\in W(A_b)$. Then there is $x\in\ell^2(\Z)$ with $\|x\|=1$ such that $\lambda=\langle A_bx,x\rangle$. Fix $m \in \Z$. Since $A_b=P_m A_b + (I-P_m) A_b = P_m A_b P_m + P_m A_b (I-P_m) + (I-P_m) A_b$, we have 
    \begin{equation*}
        \begin{split}
    \lambda &=  \langle P_m A_b P_m x,x\rangle  + \langle P_m A_b (I-P_m)x,x\rangle  + \langle (I-P_m) A_b x,x\rangle \\
    &=\langle A_b P_m x,P_m x\rangle  + \langle A_b (I-P_m) x,P_m x \rangle  + \langle A_b x,(I-P_m)x\rangle.
    \end{split}
    \end{equation*}
    Observe also that 
    \[
    \left\langle T_m \tfrac{P_m x}{\|P_m x\|},\tfrac{P_m x}{\|P_m x\|} \right\rangle = \left\langle A_b \tfrac{P_m x}{\|P_m x\|},\tfrac{P_m x}{\|P_m x\|} \right\rangle.
    \]
    The two previous equations then give
    \begin{align*}
        \lambda-\left\langle T_m\tfrac{P_m x}{\|P_m x\|},\tfrac{P_m x}{\|P_m x\|} \right\rangle 
            = {} &  \langle A_b P_m x,P_m x\rangle  -\left\langle A_b \tfrac{P_m x}{\|P_m x\|},\tfrac{P_m x}{\|P_m x\|} \right\rangle \\
           & + \langle A_b (I-P_m) x,P_m x \rangle  + \langle A_b x,(I-P_m)x\rangle \\
           = {} &  \left(1-\tfrac{1}{\| P_m x \|^2} \right) \langle A_b P_m x,P_m x\rangle  \\
           & + \langle A_b (I-P_m) x,P_m x \rangle  + \langle A_b x,(I-P_m) x\rangle.
        \end{align*}
    Hence, since $\|P_m x \| \leq \|x \| = 1$, we obtain
    \begin{align*}
        \left| \lambda-\left\langle T_m\tfrac{P_m x}{\|P_m x\|},\tfrac{P_m x}{\|P_m x\|} \right\rangle \right| 
        \leq {} & \left|1-\tfrac{1}{\| P_m x \|^2} \right| \  \left| \langle A_b P_m x, P_m x\rangle \right| \\
        & + \left| \left\langle A_b (I-P_m) x, P_m x\right\rangle \right| +
             \left| \langle A_b x, (I-P_m) x\rangle\right|\\
         \leq {} & \left|1-\tfrac{1}{\|P_m x\|^2}\right| \ \| A_b \| + \| A_b \| \  \|(I-P_m) x  \| \\
         & +\| A_b \| \  \|(I-P_m)x\|\\
         &\to 0 \text{ as } m\to -\infty,
    \end{align*}
    since $\|P_m x \| \to 1$ and $\|(I-P_m) x \| \to 0$ as $m \to -\infty$. 
    
    Therefore, given any $\epsilon>0$, there is a negative integer $k$ and $\mu\in W(T_{kn})=W(T)$ such that $|\lambda-\mu|<\epsilon$. Thus $W(A_b)\subseteq \overline{W(T)}.$
\end{proof}

\section{Main results}

In this section we state and prove the main results of this paper. Applications to a specific alphabet will be given in the next section.  Our goal is to realize the closure of the numerical range of a  periodic tridiagonal operator as the closure of the convex hull of the union of numerical ranges of its symbol matrices. For this we follow closely the work of Bebiano et al.~\cite{BebianoEtAl}. First we will observe that a circulant matrix associated to a periodic tridiagonal operator is similar to the matrix resulting from the direct sum of its associated matrix symbols.  

Recall that $n\in \N$ and $s>1$ are given integers and $m=s(n+1)$. For each $0\leq k< s$, let $\phi_k=\dfrac{2\pi k}{s}$ and denote $\rho_k=\exp(i\phi_k)$
Let us define, for each $0\leq j \leq n$, the $m$-vector 
\begin{equation}\label{ujk}
    u_{j,k}=\tfrac{1}{\sqrt s}(\underbrace{0, \dots, 0}_{j}, 1, \underbrace{0, \dots, 0}_{n}, \rho_k, \underbrace{0, \dots, 0}_{n}, \rho_{k}^{2}, \underbrace{0, \dots, 0}_{n}, \dots,
    \rho_{k}^{s-1}, \underbrace{0, \dots, 0}_{n-j})^T
\end{equation} 

The following lemma is straightforward computation. We will omit some of the details in the  proof.

\begin{lemma}
Let $n \in \N$, let $s \in \N$, $s>1$, and define $m=s(n+1)$. The set $\{ u_{j,k} \in \C^m \, : \, 0\leq j \leq n, \ 0\leq k <s \}$, where $u_{j,k}$ is defined as in Equation~\eqref{ujk}, is an orthonormal basis for $\C^m$.
\end{lemma}
\begin{proof}
    It is clear that $\|u_{j,k}\|=1$ and a computation shows that if $k\neq l$ then
    \begin{align*}
        \langle u_{j,k} , u_{j,l}\rangle &= 
        1+\rho_k\overline{\rho_l}+\rho_k^2\overline{\rho_l^2}+\cdots+
        \rho_k^{s-1}\overline{\rho_l^{s-1}}\\
        &=
        1+e^{i\frac{2\pi}{s}(k-l)}+\left(e^{i\frac{2\pi}{s}(k-l)}\right)^2+\cdots+
        \left(e^{i\frac{2\pi}{s}(k-l)}\right)^{s-1}\\
        &=0.
    \end{align*}
    Also, it is clear that $\langle u_{i,k} , u_{j,l}\rangle=0$ if $i\neq j$.
\end{proof}

The following theorem shows that the matrix $C_m$, which will play an important role in what follows, can be written as a direct sum of simpler parts.

\begin{theorem}\label{similar}
Let $n, s \in \N$ and let $s >1$. For $m=s(n+1)$, let $C_m$ be the circulant matrix associated to a tridiagonal $(n+1)$-periodic operator $T$. For $0\leq k <s$, let $\phi_k=\frac{2\pi k}{s}$. Then $C_m$ is unitarily equivalent to a block diagonal matrix where the main diagonal blocks are the symbols $T_{\phi_k}$. More precisely, there exists a unitary matrix $U\in M_{m}(\C)$ such that
\begin{equation*}
    C_m=U\left(T_{\phi_0}\oplus T_{\phi_1}\oplus\cdots\oplus T_{\phi_{s-1}} \right)U\sp\ast .
\end{equation*}
\end{theorem}
\begin{proof}
  Recall that $\rho_k=\exp(i \phi_k)$. A computation shows that, for $0\leq k <s$,
  \begin{align*}
      C_m\, u_{0,k}&=b_0u_{0,k}+a_1u_{1,k}+c_n{\rho_k}u_{n,k},\\
      C_m\, u_{n,k} &=\rho_k^{s-1} a_0 u_{0,k}+c_{n-1}u_{n-1, k}+b_n u_{n, k},
 \end{align*}
  and for $1\leq j\leq n-1$
  \begin{equation*}
      C_m\, u_{j,k}=c_{j-1}u_{j-1,k}+b_ju_{j,k}+a_{j+1}u_{j+1,k}
  \end{equation*}
  These equations can be written succinctly as the matrix equation
  \[
  \begin{pmatrix} C_m u_{0,k} \\  C_m u_{1,k} \\  C_m u_{2,k} \\ \vdots\\ \vdots \\ C_m u_{n-1,k} \\   C_m u_{n,k} \end{pmatrix}
  = \begin{pmatrix} 
  b_0 & a_1 &  0   &    &                   & 0 & c_n \rho_k\\
  c_0 & b_1 & a_2  & 0  & &  & 0\\
    0 & c_1 & b_2 & a_3 & 0 \\
     &   & \ddots & \ddots & \ddots \\
     &   &  & \ddots & \ddots & \ddots \\
    0 &     &         & 0 &  c_{n-2} & b_{n-1} & a_n \\
 \rho_k^{s-1} a_0& 0   &   &    & 0  &  c_{n-1} & b_n 
  \end{pmatrix}
  \begin{pmatrix} u_{0,k} \\  u_{1,k} \\  u_{2,k} \\ \vdots\\ \vdots\\ u_{n-1,k} \\   u_{n,k} \end{pmatrix},
  \]
  where the entries in each {\em column vector} are vectors in $\C^m$.
  
  We now define the unitary matrix $U$ as the $m\times m$ matrix with columns given by
  \[
  \begin{pmatrix} 
  \vrule & \vrule & & \vrule &\vrule & \vrule & & \vrule & & &\vrule & \vrule &  &\vrule  \\
  u_{0,0} & u_{1,0} & \dots & u_{n,0} & u_{0,1} & u_{1,1} &\dots & u_{n,1} & \dots & \dots & u_{0,s-1} & u_{1,s-1} &\dots & u_{n,s-1} \\
    \vrule & \vrule & & \vrule &\vrule & \vrule & & \vrule & & & \vrule & \vrule & & \vrule
  \end{pmatrix}.
  \]
  A straightforward argument now shows that $U^* C_m U$ is a block diagonal matrix, with blocks the $(n+1) \times (n+1)$ matrices $T_{\phi_k}$, for $0\leq k < s$. That finishes the proof.
\end{proof}

The following proposition is probably well-known to the experts; however, since we are unable to find a reference in the literature, we include a proof here for completeness sake.

\begin{proposition}\label{W(T)=unionW(Tn)}
Let $T$ be an operator on $\ell^2(\N_0)$ and let $T_k$ be the $k\times k$ matrix which is the compression of $T$ to the subspace formed by the first $k\geq 1$ components of $\ell^2(\N_0)$. Then
\[
\bigcup_{k=1}^{\infty}W(T_k) \subseteq W(T) \subseteq \overline{\bigcup_{k=1}^{\infty}W(T_k)}.
\]
and hence,
\[
    \overline{W(T)}=\overline{\bigcup_{k=1}^{\infty}W(T_k)}.
\]
\end{proposition}
\begin{proof}
   Since for each $1\leq k <\infty$ we have that $T_k$ is a compression of $T$, it then follows that 
   \[
   W(T_1) \subseteq W(T_2) \subseteq W(T_3) \subseteq W(T_4) \subseteq \cdots \subseteq W(T)
   \]
   and so  $\bigcup_{k=1}^{\infty}W(T_k) \subseteq W(T)$.

   For the other inclusion, let $\lambda$ belong to $W(T)$ so that $\lambda=\inpr{Tx,x}$ for some $x\in\ell^2(\N_0)$ with $\|x\|=1$. Let us denote by $\vec{x}_k$ the vector in $\C^k$ consisting of the first $k$-components of $x$. Then we have
   
  \begin{align*} 
      \left|\lambda -\inpr{T_k\left(\frac{\vec{x}_k}{\|\vec{x}_k\|}\right),\frac{\vec{x}_k}{\|\vec{x}_k\|} }\right|  
    &\leq
      \left|\lambda - \inpr{T_k\vec{x}_k,\vec{x}_k} \right| 
        + \left| \inpr{ T_k\vec{x}_k,\vec{x}_k }  
            -\inpr{T_k\left(\frac{\vec{x}_k}{\|\vec{x}_k\|}\right),\frac{\vec{x}_k}{\|\vec{x}_k\|} }\right| \\
    &= \left|\lambda - \inpr{T_k\vec{x}_k,\vec{x}_k} \right|
     +  \left| 1-\frac{1}{\|\vec{x}_k\|^2} \right| \ \left| \inpr{T \vec{x}_k, \vec{x}_k} \right| \\
    &\to 0 \text{ as } k\to\infty, 
  \end{align*}
  since $\vec{x}_k \to x$ (if we allow ourselves to think of $\vec{x}_k$ as a vector in $\ell^2(\N_0)$) and hence $\inpr{T_k \vec{x}_k, \vec{x}_k} \to  \inpr{T x, x}= \lambda$.
  
  Hence, for each $\epsilon >0$, there exists $K\in \N$ and $\mu \in W(T_K) \subseteq \bigcup_{k=1}^{\infty}W(T_k)$ such that $| \lambda - \mu|<\epsilon$. Thus $\lambda\in\overline{\bigcup_{k=1}^{\infty}W(T_k)}$, as was to proved.
\end{proof}

We now establish one of the inclusions in our main result as a consequence of our previous work.

\begin{corollary} \label{OneInclusion}
Let $n\in \N$, let $T$ be a $(n+1)$-periodic tridiagonal operator and let $T\sb\phi$ be the symbol of $T$. Then
\begin{equation*}
    \overline{W(T)}
    \subseteq \overline{\conv{\bigcup_{\phi\in[0,2\pi)}W(T_{\phi_k})}}
\end{equation*}
\end{corollary}
\begin{proof}
First of all, as was done in the previous proposition, we denote by $T_k$ the compression of the operator $T$ to the subspace formed by the first $k$ components of $\ell^2(\N_0)$. 

For every positive integer $s>1$, let $m=s(n+1)$. We notice that by removing the last row and the last column in $C_m$ we obtain the matrix $T_{m-1}$ and so $W(T_{m-1})\subseteq W(C_m)$. Applying Theorem~\ref{similar} we get $W(C_m)=\conv{\bigcup_{k=0}^{s-1} W(T_{\phi_k})}$ and therefore $W(T_{m-1})\subseteq \conv{\bigcup_{k=0}^{s-1}W(T_{\phi_k})}$. 

Now, clearly
\[
\bigcup_{k=0}^{s-1}W(T_{\phi_k}) \subseteq \bigcup_{\phi\in[0,2\pi)}W(T_{\phi}),
\]
and thus it follows that $W(T_{m-1})\subseteq \conv{\bigcup_{\phi\in[0,2\pi)}W(T_{\phi})}$. Hence, since
\[
W(T_1) \subseteq W(T_2) \subseteq W(T_3) \subseteq \dots,
\]
we obtain  
\[
\bigcup_{k=1}^{\infty}W(T_k) \subseteq \conv{\bigcup_{\phi\in[0,2\pi)} W(T_{\phi_k})}.
\]
Applying now Proposition~\ref{W(T)=unionW(Tn)}, we obtain the desired result.
\end{proof}

For the next theorem, we first establish the following lemma.

\begin{lemma}\label{eigenvector}
Let $n, s \in \N$, $s>1$, and let $m=s(n+1)$. For each $k=0, 1, 2, \dots, s-1$, if $\lambda$ is an eigenvalue of $T_{\phi_k}$ with eigenvector $\vec{v}=(v_0,v_1,\ldots,v_n)$, then $\lambda$ is an eigenvalue of the circulant matrix $C_m$ with eigenvector
\[
\vec{v}_{\phi_k}=(v_0, \dots, v_n, v_0 e^{i\phi_k}, \dots, v_n e^{i\phi_k}, \dots, v_0 e^{i(s-1) \phi_k}, \dots, v_n e^{i(s-1)\phi_k}).
\]
Conversely, given an eigenvalue $\lambda$ of $C_m$, there is some $k=0, 1, 2, \dots, s-1$, such that $\lambda$ has an eigenvector of the form $\vec{v}_{\phi_k}$ as above.
\end{lemma}
\begin{proof}
    A straightforward computation shows that $\vec{v}_{\phi_k}$ is indeed an eigenvector for the eigenvalue $\lambda$ of $C_m$ if $\vec{v}$ is eigenvector for the eigenvalue $\lambda$ of  $T_{\phi_k}$. 
    
    For the last assertion, as a consequence of Theorem~\ref{similar}, we have that \[
    \sigma(C_m)=\bigcup_{k=0}^{s-1}\sigma(T_{\phi_k}).
    \]
    Therefore, if $\lambda\in\sigma(C_m)$ is given, there is $k$ such that $\lambda\in\sigma(T_{\phi_k})$. Let $\vec{v}$ an eigenvector for $T_{\phi_k}$ corresponding to the eigenvalue $\lambda$. Then the corresponding vector $\vec{v}_{\phi_k}$ does the job.  
\end{proof}

\begin{theorem}\label{th:selfadjoint}
 Let $n \in \N$ and let $T$ be a tridiagonal $(n+1)$-periodic selfadjoint operator and let $T\sb\phi$ be the symbol of $T$ (which is a Hermitian matrix). Let $\lambda^-(\phi)$ and $\lambda^+(\phi)$ denote the smallest and largest eigenvalues of $T\sb\phi$ and let $a=\min_{\phi\in[0,2\pi)}\lambda^-(\phi)$ and $b=\max_{\phi\in[0,2\pi)}\lambda^+(\phi)$. Then
\[
\overline{W(T)}=[a,b].
\]
\end{theorem}
\begin{proof}
We start by showing that $[a,b] \subseteq \overline{W(T)}$. Let $s\in \N$, $s>1$. Let $\lambda_s \in \sigma(C_{s(n+1)})$. Hence, by Theorem~\ref{similar}, there exists $k=0, 1, 2, \dots, s-1$ such that  $\lambda_{s} \in\sigma(T_{\phi_k})$. Choose an eigenvector $\vec{v}=(v_0,v_1,\ldots,v_n)$ of $T_{\phi_k}$ of norm $\frac{1}{\sqrt{s}}$. Then the vector $\vec{v}_{\phi_k}$ is an eigenvector of norm $1$ for the circulant matrix $C_{s(n+1)}$, in the manner of Lemma~\ref{eigenvector}. 

If $T_{s(n+1)}$ denotes the compression of the operator $T$ to the subspace of $\ell^2(\N_0)$ consisting of the first $s(n+1)$ coordinates, then
\begin{align*}
    \inpr{ T_{s(n+1)}\vec{v}_{\phi_k},\vec{v}_{\phi_k}}&=
    \inpr{ C_{s(n+1)}\vec{v}_{\phi_k},\vec{v}_{\phi_k}}
    + \inpr{\left(T_{s(n+1)}-C_{s(n+1)}\right)\vec{v}_{\phi_k},\vec{v}_{\phi_k}}\\
    &= \lambda_{s} + \inpr{\left(T_{s(n+1)}-C_{s(n+1)}\right)\vec{v}_{\phi_k}, \vec{v}_{\phi_k}}.
\end{align*}
Since $\inpr{ T_{s(n+1)}\vec{v}_{\phi_k},\vec{v}_{\phi_k}} \in W(T_{s(n+1)}) \subseteq W(T)$, we have that 
\[
\lambda_{s} + \inpr{\left(T_{s(n+1)}-C_{s(n+1)}\right)\vec{v}_{\phi_k}, \vec{v}_{\phi_k}} \in W(T).
\]
On the other hand we have
\begin{align*}
    \left|\inpr{\left(T_{s(n+1)}-C_{s(n+1)}\right)\vec{v}_{\phi_k}, \vec{v}_{\phi_k}} \right| & =
    \left|-a_0 v_n\overline{v_0} e^{i(s-1)\phi_k} -c_n v_0 \overline{v_n} e^{-i(s-1) \phi_k} \right|\\
    &\leq \left( |a_0|+|c_n|\right) \, |v_0| \, |v_n| \\
        &\leq \frac{1}{s} \left( |a_0|+|c_n| \right),
\end{align*}
since $ |v_i| \leq \frac{1}{\sqrt{s}}$ for each $i=0,1,2,\dots,n$.

Hence, if we can show that the {\em sequence of eigenvalues} $(\lambda_{s})$ converges to some number as $s \to \infty$, then we will have
\begin{equation}\label{lim_in_W}
\lim_{s\to\infty}\lambda_{s} \in \ovl{W(T)}.
\end{equation}
To assure convergence, we will focus on a particular choice of eigenvalue $\lambda_{s}$. For each $s$, let $\lambda_s^-$ and $\lambda_s^+$ denote the smallest and largest eigenvalues of $C_{s(n+1)}$, respectively. 
We will show that 
\begin{equation}\label{lim=min}
     \lim_{s\to\infty}\lambda_s^-=\min_{\phi\in[0,2\pi)}\lambda^-(\phi)
\end{equation}
where $\lambda^-(\phi)$ is the smallest eigenvalue of $T\sb\phi$. Hence, by taking the limit as $s\to\infty$, we will obtain from the computations above, that $a=\min_{\phi\in[0,2\pi)}\lambda^-(\phi)\in \ovl{W(T)}$.

Since $\lambda_s^-$ is the smallest eigenvalue of $C_{s(n+1)}$ and since, by Theorem~\ref{similar} we have $\sigma(C_{s(n+1)})=\bigcup_{k=0}^{s-1}\sigma(T\sb{\phi_k})$, then $\lambda_s^-$ is not only an eigenvalue of $T\sb{\phi_k}$ for some $k$ but in fact it is the smallest eigenvalue among the eigenvalues of all symbols $T_{\phi_0},T_{\phi_1},\ldots,T_{\phi_{s-1}}$; i.e., $\lambda^-_s=\lambda^-(\phi_k)$ for some $k=0,1,2, \dots, s-1$.

Let 
\[ 
\lambda^-(\phi\sp\ast)=\min_{\phi\in[0,2\pi)}\lambda^-(\phi)
\]
where $\phi\sp\ast$ is the point where the minimum is reached (here we are using the continuity of $\lambda^-(\phi)$; see, for example, \cite[p. 108-109]{Kato}).
Therefore $\lambda^-(\phi\sp\ast)\leq\lambda^-(\phi\sb{k})=\lambda_s^-$. Using continuity of $\lambda^-(\cdot)$, for all $\epsilon>0$ there is $\delta>0$ such that $|\phi-\phi\sp\ast|<\delta$ implies $|\lambda^-(\phi)-\lambda^-(\phi\sp\ast)|<\epsilon$. Then it follows that $\lambda^-(\phi\sp\ast)\leq\lambda^-(\phi)<\lambda^-(\phi\sp\ast)+\epsilon$. Now, by the density of the rational multiples of $2\pi$ in the interval $[0,2\pi)$, there exists $N\in \N$ such that for $s\geq N$ we may choose $0\leq r\leq s-1$ such that $\frac{2\pi r}{s}\in(\phi\sp\ast-\delta,\phi\sp\ast+\delta)$. Thus for all $s\geq N$, 
\[
  \lambda^-(\phi\sp\ast)\leq\lambda_s^-=\lambda^-(\phi_k)\leq\lambda^-(\phi_{r})<\lambda^-(\phi\sp\ast)+\epsilon.
\]
This proves Equation~\eqref{lim=min}. Since it is analogous, we omit the proof of
\[
\lim_{s\to\infty}\lambda_s^+=\max_{\phi\in[0,2\pi)}\lambda^+(\phi),
\]
where $\lambda_s^+$ is the largest eigenvalue of $C_{s(n+1)}$ and $\lambda^+(\phi)$ is the largest eigenvalue of $T_\phi$. 

Therefore, by Expression~\eqref{lim_in_W}, we have proved
\[
  a= \min_{\phi\in[0,2\pi)}\lambda^-(\phi)=\lim_{s\to\infty}\lambda_s^-\in\overline{W(T)}
\]
and 
\[
  b= \max_{\phi\in[0,2\pi)}\lambda^+(\phi)=\lim_{s\to\infty}\lambda_s^+\in\overline{W(T)}.
\]
Thus  $[a,b]\subseteq \overline{W(T)}$.

Now, to show that $\overline{W(T)} \subseteq [a,b]$, we first observe that
\begin{align*}
    W(T_{s(n+1)-1})&\subseteq W(C_{s(n+1)})\\
            &=\conv{\bigcup_{k=0}^{s-1} W(T\sb{\phi_k})}\\
            &=\conv{\bigcup_{k=0}^{s-1} [\lambda^-(\phi_k), \lambda^+(\phi_k)]}\\
            &\subseteq \left[ \min_{\phi\in[0,2\pi)} \lambda^-(\phi),\max_{\phi\in[0,2\pi)} \lambda^+ (\phi)\right].
\end{align*}
Since we have
\[
W(T_1) \subseteq W(T_2)\subseteq W(T_3) \subseteq \cdots 
\]
we obtain, by Proposition~\ref{W(T)=unionW(Tn)}, that
\[
  \ovl{W(T)}=\ovl{\bigcup_{r=1}^{\infty}W(T_r)} \subseteq [a,b],
\]
as desired.
\end{proof}

The following lemma is well known, but since we have not been able to find a reference, we include its proof here. 


\begin{lemma}\label{reduce}
Let $A$ and $B$ be  operators on $\ell^2(\N_0)$. If for all $\theta\in[0,2\pi)$ we have  $W(\Re(e^{-i\theta} A))
 \subseteq \overline{W(\Re((e^{-i\theta}B)}$ then
 $W(A)\subseteq\overline{W(B)}$.
\end{lemma}
\begin{proof}
Let $\lambda\in W(A)$. Then for all $\theta \in [0,2\pi)$ we have $e^{-i \theta} \lambda \in W(e^{-i\theta} A)$ and hence 
\[
\Re(e^{-i\theta}\lambda) \in \Re(W(e^{-i\theta}A))=W(\Re(e^{-i\theta}A)) \subseteq \ovl{W(\Re(e^{-i\theta} B))}.
\]
    
We claim that $\lambda\in\overline{W(B)}$. If not, there is $\theta\in[0,2\pi)$ such that the supporting line of $W(B)$  perpendicular to the ray $\{re^{i\theta}\colon r>0\}$ separates $\lambda$ from $\overline{W(B)}$. Then $\lambda e^{-i\theta}$ is separated from $e^{-i\theta}\overline{W(B)}=\overline{e^{-i\theta}W(B)}$ by a vertical line $e^{-i\theta}\ell$ and so $\Re(e^{-i\theta}\lambda)\not\in\Re\left(\overline{e^{-i\theta} W(B)}\right)=\overline{\Re(W(e^{-i\theta}B))}=\overline{W(\Re(e^{-i\theta} B))}$, contradicting what we just proved.
\end{proof}

For the next theorem, observe that if $T$ is a tridiagonal periodic operator and $T_\phi$ is its symbol, then $\Re(e^{-i \theta} T)$ is also a tridiagonal periodic operator and its symbol is $\Re(e^{-i \theta} T_\phi)$

\begin{theorem}\label{maintheorem}
 Let $n \in \N$ and let $T$ be a tridiagonal $(n+1)$-periodic operator. If $T\sb\phi$ is the symbol of $T$ then
 \begin{equation*}
    \overline{W(T)}
    = \overline{\mathrm{co}\left(\bigcup_{\phi\in[0,2\pi]} W(T\sb\phi)\right)}. 
\end{equation*}
\end{theorem}
\begin{proof}
In view of Corollary~\ref{OneInclusion}, it will suffice to show that ${W(T\sb\phi)}$ is a subset of $\overline{W(T)}$  and for this purpose we will show that the conditions of Lemma~\ref{reduce} hold, that is,  for each $\theta\in[0,2\pi)$,  we will show that
$\Re(W(e^{-i\theta}T\sb\phi))\subseteq\ovl{\Re(W(e^{-i\theta}T))}.$

Fix $\theta\in[0,2\pi)$. Notice that, for each $\phi \in [0,2\pi)$, we have that $\Re(e^{-i\theta}T\sb\phi)$ is a Hermitian matrix and so we may denote by $\lambda^-(\theta,\phi)$ and $\lambda^+(\theta,\phi)$ its smallest and largest eigenvalue, respectively.
Then if $a(\theta):=\min_{\phi\in[0,2\pi)}\lambda^-(\theta,\phi)$ and $b(\theta):=\max_{\phi\in[0,2\pi)}\lambda^+(\theta,\phi)$, we have  
\begin{eqnarray*}
\mathrm{Re}\left(W\left(e^{-i\theta}T\sb\phi\right)\right)
   &=& W\left(\mathrm{Re}(e^{-i\theta}T\sb\phi)\right)\\
   &=& \left[\lambda^-(\theta,\phi),\lambda^+(\theta,\phi)\right]\\
   &\subseteq& [a(\theta),b(\theta)]\\
   &=& \overline{W(\Re(e^{-i\theta}T})\\
    &=& \overline{\Re(W(e^{-i\theta}T)},
\end{eqnarray*}
where the next-to-last equality follows from Theorem~\ref{th:selfadjoint}, since $\Re(e^{-i \theta} T)$ is a tridiagonal $(n+1)$-periodic selfadjoint operator with symbol $\Re(e^{-i\theta}T_\phi)$. It follows that $\mathrm{Re}\left(e^{-i\theta}\ W(T\sb\phi) \right) \subseteq \ovl{\Re\left(e^{-i\theta}\ W(T) \right)}$,  as was to be proved. 
\end{proof}

As mentioned in the introduction, the proof of Theorem~\ref{maintheorem} is based closely in ideas contained in \cite{BebianoEtAl}. However, their main result \cite[Theorem~3.2]{BebianoEtAl} applies to banded biperiodic Toeplitz operators and therefore does not apply to our tridiagonal periodic operators. On the other hand, just as for  \cite[Theorem~3.2]{BebianoEtAl}, Theorem~\ref{maintheorem} is also a particular case of \cite[Theorem~1]{BS}, in spite of this, we presented the proof above since we believe it leads to interesting results for tridiagonal operators.


\section{The 2-periodic case}

In this section we specialize the theory above to the case of tridiagonal operators  associated to infinite sequences with symbols in the alphabet $\mathcal{A}=\{0,1\}$.
In particular, we focus on the $(n+1)$-periodic tri\-dia\-go\-nal operators  $T=T(a,0,1)$, where $0$ and $1$ are the constant sequences of zeroes and ones, respectively.  
From \cite[Theorem~2.5]{HI-O2016}, we already  know that the numerical range of  $T$ is contained in the set $\Gamma$ equal to the convex hull of the union of the numerical ranges of the tri\-dia\-go\-nal operators $T(0,0,1)$ and $T(1,0,1)$. In fact, by  \cite[Corollary~2.7]{HI-O2016}, this set $\Gamma$ is the numerical range of some tridiagonal operator.  When  $a$ is the $2$-periodic sequence of period word $01$, we prove in this section a similar result: the closure of the numerical range of $T$ is the closure of the convex hull of the union of the numerical ranges of two matrices in $M_2(\C)$. We will conjecture the corresponding general result when $a$ is the $(n+1)$-periodic sequence with period word $0^n1$.


\begin{definition}\label{gammaphi}
Let $\phi\in[0,2\pi)$ and  let $w=1+e^{\phi i}$. Denote by $\gamma\sb\phi$ the ellipse with focal points at $\pm\sqrt{w}$ and major axis of length $1+|w|$. The  closed  convex set with boundary $\gamma\sb\phi$ will be denoted by $E\sb\phi$.
\end{definition}

Observe that $\Re(w)=1+\cos(\phi)$, $\Im(w)=\sin(\phi)$ and $|w|^2=2\Re(w)$. We will use these identities in what follows.

A characterization  of the points of the  ellipse $\gamma\sb\phi$ defined above will be useful in the sequel. We present it next.

\begin{lemma}\label{ellipse}
Consider the ellipse $\gamma\sb\phi$ given in Definition~\ref{gammaphi}. Suppose that $\theta=\arg\sqrt{w}$ and choose $\theta=0$ if $w=0$. Then the set of all points in $\gamma\sb\phi$ can be parametrized as
\begin{equation*}
\left\{ z \in \C \, : \, z=e^{\theta i} \left(\frac{|w|}{2} e^{- t i} + \frac{1}{2} e^{ti} \right) , \text{ for } t\in [0, 2\pi) \right\}
\end{equation*}
\end{lemma}
\begin{proof}
An ellipse with foci at $\left(\pm|\sqrt{w}|,0\right)$ and major axis length $1+|w|$ has minor axis length $\left|1-\left|w\right|\right|$. Thus, the parametric equations such an ellipse satisfy are 
\begin{align*}
    x=\frac{1+|w|}{2}\cos t &= \frac{1+|w|}{4}\left(
    e^{ti}+e^{-ti}\right),\\ 
    y=\frac{1-|w|}{2}\sin t &= 
    \frac{1-|w|}{4i}\left( e^{ti}-e^{-ti} \right),
\end{align*}
for $t \in [0,2\pi)$. It follows that the set of complex numbers of the form $z= \frac{|w|}{2} e^{-t i} + i\frac{1}{2} e^{ti}$ is precisely the mentioned ellipse with foci at $(\pm|\sqrt{w}|,0)$ and major axis $1+|w|$. After rotating by $e^{\theta i}$ we obtain the ellipse $\gamma_\phi$, as desired.
\end{proof}

For the next lemma, we need to make the following observation. If $\ell_\psi$ is the tangent line to the circle 
\[
\Ci=\left\{ z \in \C \, : \, |z-1|=\frac{1}{2}\right\}
\]
at the point $1+\frac{1}{2} e^{\psi i}$, then a computation shows that it has equation $x\cos\psi + y \sin\psi = \frac{1}{2}+\cos\psi$. Thus
\[
\ell_\psi=\left\{ z \in \C \, \colon \, \Re(z e^{-\psi i})=\frac{1}{2}+\cos(\psi) \right\}.
\]
Therefore, the line $\ell_\psi$ defines a partition of the complex plane into semiplanes. For $\psi \in \left[0,\frac{1}{2}\pi\right] \cup \left[\frac{3}{2} \pi, 2\pi\right)$, we define
\[
H_\psi:= \left\{ z \in \C \, \colon \, \Re(z e^{-\psi i}) < \frac{1}{2}+\cos(\psi) \right\}
\]
to be the semiplane that contains the origin. We have the following lemma.

\begin{lemma}\label{semiplanes}
Consider the ellipse $\gamma\sb\phi$ given in Definition~\ref{gammaphi}. Let  $\psi \in \left[0,\frac{1}{2}\pi\right] \cup \left[\frac{3}{2} \pi, 2\pi\right)$ and let $\ell_\psi$ and $H_\psi$ be as above. Then for every $\phi \in [0, 2\pi)$, we have $\gamma_\phi \subseteq \ovl{H_\psi}$. Furthermore,
\begin{itemize}
    \item if $\psi\not = \frac{1}{2}\pi$ and $\psi\not = \frac{3}{2}\pi$, then $\phi\not=\psi$ implies $\gamma_\phi \subseteq H_\psi$. If $\phi=\psi$ then $\ell_\psi$ is tangent to $\gamma_\phi$ at the point $1+\frac{1}{2}e^{\psi i} \in \Ci$;
    \item if $\psi=\frac{1}{2}\pi$, then $\ell_\psi$ is tangent to  $\gamma\sb\phi$ at the point $\sin\phi+i\frac{1}{2}\pi$; and
    \item if $\psi=\frac{3}{2}\pi$, then $\ell_\psi$ is tangent to  $\gamma\sb\phi$ at the point $-\sin\phi-i\frac{1}{2}\pi$.
\end{itemize}
\end{lemma}
\begin{proof}
We claim that if $z \in \gamma_\phi$, then $\Re(z e^{-\psi i}) \leq \frac{1}{2} + \cos\psi$.

Indeed, given $z \in \gamma_\phi$, using Lemma~\ref{ellipse} we obtain $z=e^{\theta i} \left(\frac{|w|}{2} e^{-t i} + \frac{1}{2} e^{ ti} \right)$ for some $t \in \left[0,2\pi\right)$. Then $z e^{-\psi i} = \frac{|w|}{2} e^{(-t -\psi + \theta) i} + \frac{1}{2} e^{(t-\psi + \theta) i}$ and hence
\begin{equation}\label{ast}
\begin{split}
\Re( z e^{-\psi i})&= \frac{|w|}{2} \cos(-t -\psi + \theta) + \frac{1}{2} \cos(t-\psi + \theta)\\
&=\frac{|w|+1}{2} \cos(-\psi+\theta) \cos(t) + \frac{|w|-1}{2} \sin(-\psi+\theta) \sin(t)\\
&=A \cos(t-B)\\
&\leq A 
\end{split}
\end{equation}
where
\[
A = \sqrt{\left(\frac{|w|+1}{2} \cos(-\psi+\theta) \right)^2 + \left( \frac{|w|-1}{2} \sin(-\psi+\theta) \right)^2},
\]
and $0\leq B<2\pi $ is the angle satisfying the equations
\begin{align}\label{sincosofB}
    \cos B=\frac{|w|+1}{2A}\cos(-\psi+\theta)\hspace*{1.5cm}\sin B=\frac{|w|-1}{2A} \sin(-\psi+\theta).
\end{align}

Notice that $\Re( z e^{-\psi i}) = A \cos(t-B) \leq A$,  with equality if and only if $t=B$. So to prove the claim, we need to show that $A \leq \frac{1}{2} + \cos(\psi)$, which we now proceed to verify.

Observe that 
\begin{align*}
    A^2 &= \left(\frac{|w|+1}{2} \cos(-\psi+\theta) \right)^2 + \left( \frac{|w|-1}{2} \sin(-\psi+\theta) \right)^2 \\
    &=\frac{1+|w|^2}{4} + \frac{|w|}{2} \left( \cos^2(-\psi+\theta) - \sin^2(-\psi+\theta) \right) \\
    &= \frac{1+|w|^2}{4} + \frac{|w|}{2} \cos(2\psi -2\theta) \\
    &= \frac{1+|w|^2}{4} + \frac{|w|}{2} \left( \cos(2\psi) \cos(2\theta) + \sin(2\psi) \sin(2\theta) \right).
\end{align*}
Also, observe that $w=|w|e^{2\theta i}=|w| \cos(2\theta)+ i |w| \sin(2\theta)$ and hence
\begin{align*}
    A^2 &=  \frac{1+|w|^2}{4} + \frac{1}{2} \cos(2\psi)\Re(w) + \frac{1}{2} \sin(2\psi) \Im(w) \\
    &= \frac{1+|w|^2}{4} + \frac{1}{2} (2 \cos^2(\psi) -1) \Re(w) + \frac{1}{2} (2 \sin(\psi) \cos(\psi)) \Im(w) \\
    &= \frac{1+|w|^2}{4} - \frac{1}{2} \Re(w) +  \cos^2(\psi) \Re(w) + \sin(\psi) \cos(\psi) \Im(w).
\end{align*}
Substituting $|w|^2=2  \Re(w)$, $\Re(w)=1+\cos(\phi)$ and $\Im(w)=\sin(\phi)$ the above expression becomes
\begin{equation*}
\begin{split}
    A^2 
    &= \frac{1}{4} +  \cos^2(\psi) (1+ \cos(\phi)) + \sin(\psi) \cos(\psi)  \sin(\phi) \\ 
    &= \frac{1}{4} +  \cos(\psi) \big( \cos(\psi) + \cos(\psi) \cos(\phi) + \sin(\psi) \sin(\phi) \big)\\ 
    &=\frac{1}{4} +  \cos(\psi) \big( \cos(\psi) + \cos(\psi-\phi) \big)\\
    &= \frac{1}{4} +  \cos^2 (\psi)  + \cos(\psi) \cos(\psi-\phi)\\ 
     &\leq \frac{1}{4} +  \cos^2(\psi) +\cos(\psi) \qquad \qquad \qquad (\text{since } \cos(\psi) \geq 0, \text{ by the choice of } \psi)\\ 
      &= \left(\frac{1}{2} + \cos(\psi)\right)^2,
 \end{split}
 \end{equation*}
and thus, since, $\frac{1}{2} + \cos(\psi) >0$, we obtain 
\begin{equation}\label{A2}
A\leq \frac{1}{2}+\cos(\psi),
\end{equation}
with equality if and only if $\psi=\phi$ or $\psi=\frac{1}{2}\pi$ or $\psi=\frac{3}{2}\pi$.

Therefore, combining inequalities (\ref{ast}) and (\ref{A2}), it completes the proof of our claim  $\Re(z e^{-\psi i}) \leq \frac{1}{2} + \cos(\psi)$, with equality if and only if $t=B$, and $\psi=\phi$ or $\psi=\frac{1}{2}\pi$ or $\psi=\frac{3}{2}\pi$. Thus, if $z\in\gamma\sb\phi$ then $z\in\ovl{H\sb\psi}$. Hence $\gamma\sb\phi\subseteq \ovl{H\sb\psi}$, as was to be proved.

To prove part (1), we assume $\psi\not=\frac{1}{2}\pi$ and $\psi\not=\frac{3}{2}\pi$, and so $\psi\in\left[0,\frac{1}{2}\pi\right)\cup\left(\frac{3}{2}\pi,2\pi\right)$. Observe that then inequality (\ref{A2}) is a strict inequality if and only if $\phi\not=\psi$. Therefore, if $\phi\not=\psi$ then from inequalities (\ref{ast}) and (\ref{A2}) we obtain $\Re(z e^{-\psi i}) \leq A < \frac{1}{2} + \cos(\psi)$ and so $\gamma\sb\phi\subseteq H\sb\psi$ as wanted. 

On the other hand, if $\phi=\psi$, we are going to show that for the angle $t=B$ in Equation~(\ref{sincosofB}),  we obtain that $z=e^{\theta i}\left(\frac{|w|}{2} e^{- B i} + \frac{1}{2} e^{B i} \right)$, which is a point of $\gamma\sb\phi$ by Lemma~\ref{ellipse}, is the tangent point of $\ell\sb\phi$ to the circle $\Ci$. Therefore, we must show that $e^{\theta i}\left(\frac{|w|}{2} e^{-Bi} + \frac{1}{2} e^{Bi} \right)=1+\frac{1}{2}e^{\phi i}$. Indeed, using Equation~(\ref{sincosofB}) we compute

\begin{align*}
    e^{\theta i}\left(\frac{|w|}{2} e^{-Bi} + \frac{1}{2} e^{Bi} \right)
    &=
    \frac{e^{\theta i}}{4A}\left( (1+|w|)^2\cos(-\phi+\theta)-i(1-|w|)^2\sin(-\phi+\theta)
    \right)\\
    &=\frac{e^{\theta i}}{4A}\left( 
    \left(1+|w|^2\right)e^{-(-\phi+\theta)i}+2|w|e^{(-\phi+\theta)i}
    \right)\\
    &=\frac{1}{4A}\left(
       (1+|w|^2)e^{\phi i}+2|w|e^{(-\phi+2\theta)i}
    \right)\\
    &=\frac{1}{4A}\left(
       (2+2A)e^{\phi i}+2e^{-\phi i}w  \right) 
       \end{align*}
because $|w|^2=2\Re(w)=2+2\cos(\phi)=1+2A$ and $w=|w|e^{2\theta i}$. Now, since $w=1+e^{i\phi}$ we obtain
\begin{align*}
    e^{\theta i}\left(\frac{|w|}{2} e^{-Bi} + \frac{1}{2} e^{Bi} \right)
    &= \frac{1}{4A}\left(
       (2+2A)e^{\phi i}+2 e^{-\phi i}+2
    \right)\\
    &=\frac{1}{2A} \left(e^{\phi i} + A e^{\phi i} + e^{-\phi i} +1\right) \\
    &=\frac{1}{2A} \left(A e^{\phi i} + 2\cos(\phi) +1)\right) \\
    &=\frac{1}{2A} \left(A e^{\phi i} + 2 A)\right) \\
    &=\frac{1}{2} \left(e^{\phi i} + 2\right) \\
    &=1+\frac{1}{2}e^{\phi i},
\end{align*}
as wanted.

To prove part (2), we assume $\psi=\frac{1}{2}\pi$. Observe that in this case $\ell_\frac{\pi}{2}$ is the horizontal line through $\frac{1}{2} i$.

We are going to show that  for the angle $t=B$ in Equation~(\ref{sincosofB}),  we obtain $e^{\theta i}\left(\frac{|w|}{2} e^{- t i} + \frac{1}{2} e^{ti} \right)=\sin\phi+i\frac{\pi}{2}$, where the left hand side is a  point of $\gamma\sb\phi$ by Lemma~\ref{ellipse}, and the right hand side is the tangent point on $\ell\sb{\frac{\pi}{2}}$. Observe that from Equation~(\ref{A2}) we have $A=\frac{1}{2}$ and so from Equation~(\ref{sincosofB}) we deduce that $\cos B=(|w|+1)\sin\theta$ and $\sin B=-(|w|-1)\cos\theta$. Thus
\begin{align*}
e^{\theta i} \left(\frac{|w|}{2} e^{-B i} + \frac{1}{2} e^{B i}\right)
={} &\frac{|w|}{2} e^{(\theta-B)i}+\frac{1}{2} e^{(\theta+B)i}  \\
={} &\frac{|w|+1}{2}\cos\theta\cos B +
      \frac{|w|-1}{2}\sin\theta\sin B  \\
      & +  i\left(\frac{|w|+1}{2}\sin\theta\cos B -
      \frac{|w|-1}{2}\cos\theta\sin B
      \right)\\
={} &\frac{(|w|+1)^2}{2}\cos\theta\sin\theta-
      \frac{(|w|-1)^2}{2}\sin\theta\cos\theta \\
      & + i\left( \frac{(|w|+1)^2}{2}\sin^2\theta
      +\frac{(|w|-1)^2}{2}\cos^2\theta\right)\\
={} & 2|w|\sin\theta\cos\theta+i\left(
    \frac{|w|^2+1}{2}-|w|\cos(2\theta)
    \right)\\
={} & |w|\sin(2\theta) +i\left(
      \frac{|w|^2+1}{2}-\Re(w)
       \right)\\
={} &\Im(w)+i\left(
       \frac{3+2\cos\phi}{2}-\left( 1+\cos\phi \right)
       \right)\\
={} &\sin\phi +\frac{1}{2}i,
\end{align*}
as was to be proved.

The proof that if $\psi=\frac{3}{2}\pi$, then $\ell_\psi$ (which is the horizontal line through $-i\frac{1}{2}$) is tangent to $\gamma_\phi$ at the point $-\sin\phi-i\frac{1}{2}\pi$ is similar, so we omit it.
\end{proof}

Clearly, there is an analogous lemma for the tangent lines to the semicircle $-\left(1+\frac{1}{2} e^{\psi i}\right)$ for $\psi\in\left[0,\frac{\pi}{2}\right)\cup\left[\frac{3\pi}{2}, 2\pi\right)$: each ellipse $\gamma_\phi$ is contained in the semiplane (containing the origin) defined by each tangent and is tangent to the horizontal tangent line.

Using the previous lemma, we can see that the closure of the convex hull of the union of the ellipses $\gamma_\phi$ has a simple form.

\begin{proposition}\label{EqualConvex}
Let $\phi\in[0,2\pi)$ and  let  $E\sb\phi$ be the closed  convex set with boundary the ellipse $\gamma_\phi$, as in Definition~\ref{gammaphi}. Furthermore, let $A$ and $B$ denote the closed convex sets with boundary the circles of radii $1/2$ and centers at $-1$ and $1$, respectively. Then
\begin{equation*}
  \ovl{ \conv{\bigcup_{\phi\in[0,2\pi)} E_\phi }}
  =
  \conv{A\cup B}    
\end{equation*}
\end{proposition}
\begin{proof}
First observe that the boundary of $\conv{A \cup B}$ consists of the union of the segment between $-1+\frac{1}{2}i$ and $1+\frac{1}{2}i$, the segment between $-1-\frac{1}{2}i$ and $1-\frac{1}{2}i$, the semicircle $1+\frac{1}{2} e^{\psi i}$ for $\psi\in\left[0,\frac{\pi}{2}\right]\cup\left[\frac{3\pi}{2}, 2\pi\right)$,
and the semicircle $-\left(1+\frac{1}{2} e^{\psi i}\right)$ for $\psi\in\left[0,\frac{\pi}{2}\right]\cup\left[\frac{3\pi}{2}, 2\pi\right)$.

    To prove the first inclusion, it suffices to show that for each $\phi \in [0,2\pi)$, the ellipse $\gamma_\phi$ is contained in $\conv{A \cup B}$. But observe that, by Lemma~\ref{semiplanes} each ellipse is tangent to the segment between the points $-1+\frac{1}{2}i$ and $1+\frac{1}{2}i$, and hence is below said segment. Also, each ellipse is tangent to the segment between the points $-1-\frac{1}{2}i$ and $1-\frac{1}{2}i$ and hence is above said segment. Also, by Lemma~\ref{semiplanes}, each ellipse is contained in $H_\psi$ for each $\psi \in \left[0,\frac{1}{2}\pi\right] \cup \left[\frac{3}{2}\pi, 2\pi\right)$; i.e. each ellipse $\gamma_\phi$ is contained in the semiplane (containing the origin) determined by each tangent line to the semicircle $1+\frac{1}{2} e^{\psi i}$ for $\psi\in\left[0,\frac{\pi}{2}\right) \cup \left(\frac{3\pi}{2}, 2\pi\right)$, and hence each ellipse is to the ``left'' (or ``below'', in the cases $\psi=\frac{1}{2}\pi$ or $\psi=\frac{3}{2}\pi$) of said semicircle. Analogously, one can show that each ellipse is contained in the semiplane (containing the origin) determined by each tangent line to the semicircle $-\left(1+\frac{1}{2} e^{\psi i}\right)$ for $\psi\in\left[0,\frac{\pi}{2}\right] \cup \left[\frac{3\pi}{2}, 2\pi\right)$, and hence each ellipse is to the ``right'' (or ``above'') of said semicircle. Hence each ellipse is contained in $\conv{A \cup B}$.
    
    For the other inclusion, by convexity, we only need to argue that the boundary of the right-hand side is included in the left-hand side. In fact, we need to show that each point in the boundary of $\conv{A\cup B}$ is in at least one of the ellipses $\gamma_\phi$. But every point in each of these parts is the tangent point to an ellipse $\gamma_\phi$, as shown by Lemma~\ref{semiplanes}, and hence it belongs to $E_\phi$ for some $\phi \in [0, 2\pi)$. That concludes the proof. 
\end{proof}

The following lemma is an easy observation but it will be useful in the sequel. Compare with \cite[Lemma~1]{CS} where a more general result for matrices is proved.

\begin{lemma}\label{le:symmetric}
Let $T=T(a,0,c)$ be a tridiagonal operator where $0$ is the sequence of zeroes and $a,c$ are  sequences on some alphabet. Then $W(T)$ is symmetric with respect to the origin.
\end{lemma}
\begin{proof}
    Let $\lambda=\inpr{Tx,x}$ be an arbitrary element in $W(T)$, where $x=(x_0,\ x_1,\ x_2,\ x_3, \cdots )$ is a unitary vector. Then $y=(x_0,\ -x_1,\ x_2,\ -x_3,\cdots )$ is also a unitary vector. A straightforward computation then shows that $\inpr{Ty,y}=-\lambda$, and hence $-\lambda \in W(T)$, as was to be proved.
\end{proof}

We are now ready to say what the numerical range of the operator $T(a,b,c)$ is for a particular case. Although it is possible to deduce it from \cite[Theorem~4.1]{BebianoEtAl}, we believe our elementary approach to be of independent interest.

\begin{theorem}\label{spcase}
Let $T=T(a,0,1)$ be the tridiagonal operator where $0$  and $1$ are the constant sequences of zeroes and ones, respectively, and $a$ is the periodic sequence with period word $01$.  Let $A$ and $B$ denote the closed convex sets with boundary the circles of radii $1/2$ and centers at $-1$ and $1$, respectively.
Let $C=\begin{pmatrix}1&1\\0&1\end{pmatrix}$ and $D=\begin{pmatrix}-1&\phantom{-}1\\0&-1\end{pmatrix}$. Then 
\[
 \overline{W(T)}=\conv{  A \cup B }=\conv{W(C)\cup W(D)}.
 \]
\end{theorem}
\begin{proof}
We begin by noticing that the numerical ranges $W(C)$ and $W(D)$ are $A$ and $B$, respectively.
    Now, the result follows easily by  combining Theorem~\ref{maintheorem} with Proposition~\ref{EqualConvex}. Indeed, it will suffice to show that  $W(T\sb\phi)$ is $E\sb\phi$. Observe that by equation \eqref{eq:symbol2}, 
    \[
    T_\phi=\begin{pmatrix}0 & 1+ e^{-\phi i}\\ e^{\phi i} & 0 \end{pmatrix}.
    \]
    Put $w=1+e^{i\phi}$. Then the numerical range  $W(T\sb\phi)$ is the closed set with boundary the ellipse with foci at $\pm \sqrt{e^{\phi i} (1+e^{-\phi i})}=\pm \sqrt{w}$ and mayor axis equal to $|e^{\phi i}| +|1+e^{-\phi i}|=1+|w|$, that is, $E\sb\phi$, as wanted.
    \end{proof}
    
    There is a simpler alternative proof of the contention $\conv{A \cup B} \subseteq \ovl{W(T)}$. Indeed, by Lemma~\ref{le:symmetric}, since $A$ and $B$ are symmetric with respect to the origin, it will suffice to show that $A\subseteq \ovl{W(T)}$ in order to obtain the inclusion $\conv{A\cup B} \subseteq \ovl{W(T)}$, by convexity of the numerical range. For this purpose, let $1+\lambda /2$ be an arbitrary element in the interior of $A$, where $\lambda$ is a complex number with modulus less than one. Let $x=(1,1,\lambda,\lambda,\lambda^2,\lambda^2,\ldots)$ a vector in $\ell^2(\N_0)$ and let $u=x/\|x\|$ be the normalization of $x$. A computation shows that $\inpr{Tu,u}=1+\lambda/2$, and so $1+\lambda/2$ belongs to $W(T)$. Hence  $ A$ is included in $\overline{W(T)}$, as was to be proved.

\begin{figure}
\includegraphics[scale=.4]{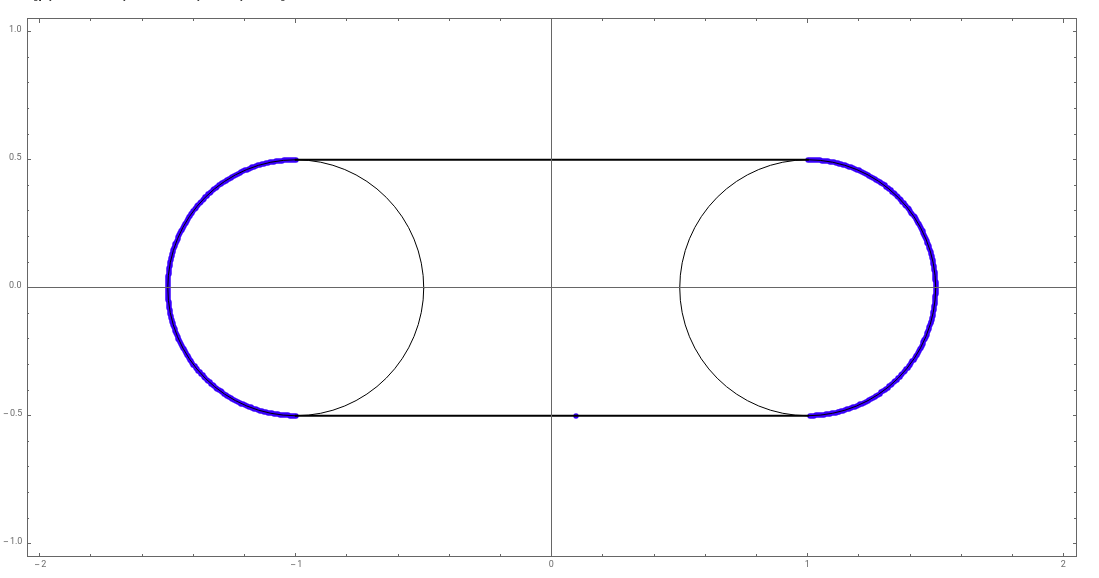}
\caption{The convex hull of the blue points represents the closure of $W(T)$ for $n=1$. The circles are the numerical ranges of $B_1+J_1$ and $B_1-J_1$, respectively.}
\label{n=1}
\end{figure}

\begin{figure}
\includegraphics[scale=.38]{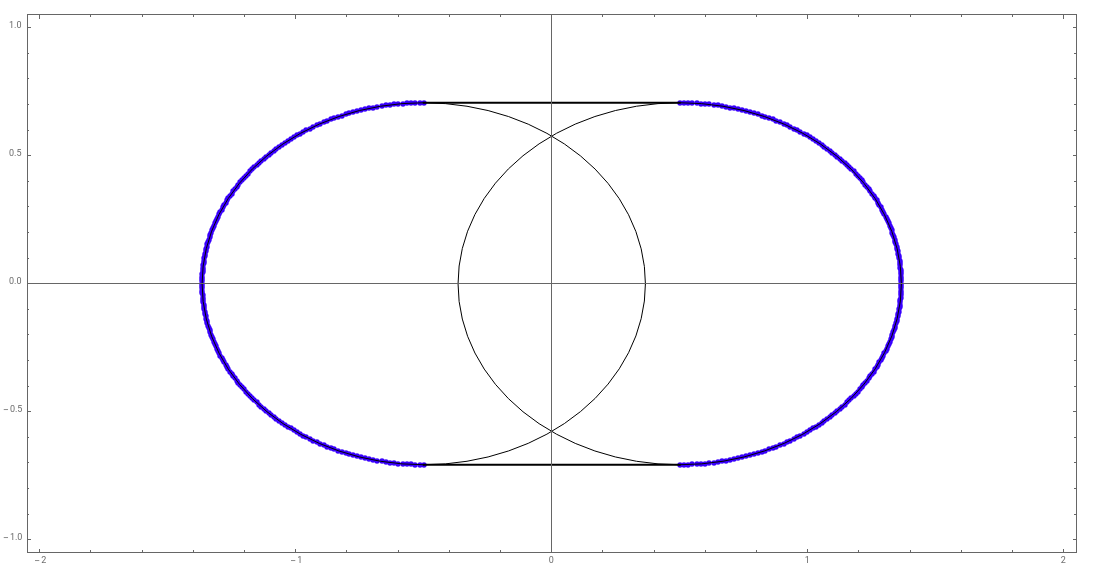}
\caption{The convex hull of the blue points represents the closure of $W(T)$ for $n=2$. The ellipses are the numerical ranges of $B_2+J_2$ and $B_2-J_2$, respectively. They are the ellipses centered at $(1/2,0)$ and $(-1/2,0)$, respectively, with major axis $\sqrt{3}$ and minor axis $\sqrt{2}$.}
\label{n=2}
\end{figure}

\begin{figure}
\includegraphics[scale=.38]{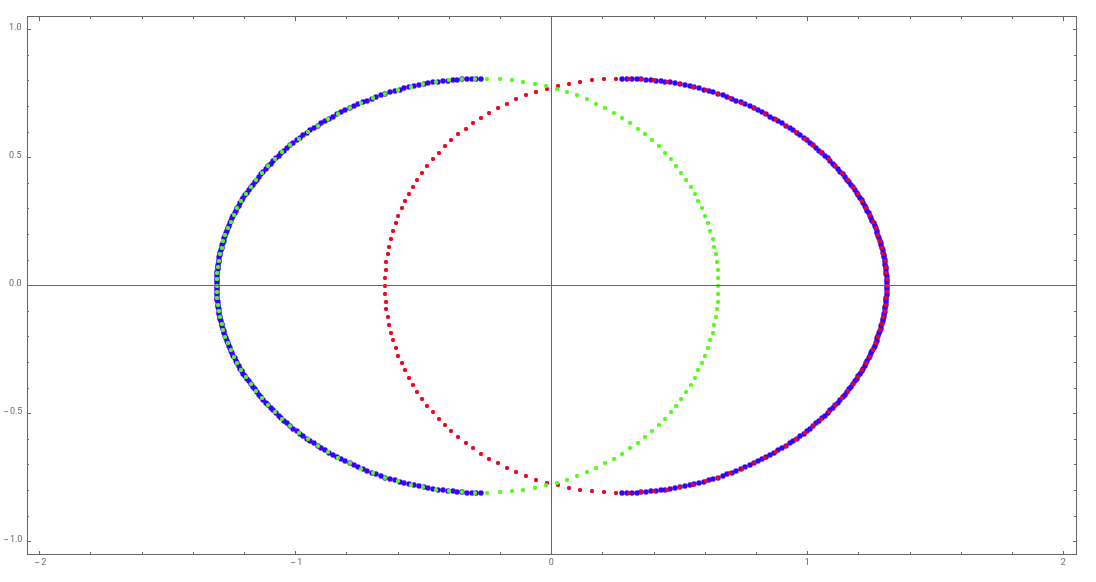}
\caption{The convex hull of the blue points represents the closure of $W(T)$ for $n=3$. The red points represent the numerical range of $B_3+J_3$ and the green points that of $B_3-J_3$, respectively.}
\label{n=3}
\end{figure}

\begin{conjecture}\label{conjecture}
Let $T=T(b,0,1)$ be the tridiagonal operator where $0$  and $1$ are the constant sequences of zeroes and ones, respectively, and $b$ is the $(n+1)$-periodic sequence with period word $0^{n} 1 $. Let $J_n$ denote the $(n+1)\times (n+1)$ matrix with value $1$ at the positions $(1,1)$ and $(n+1,n+1)$ and zero everywhere else, and let $B_n$ be the $(n+1)\times (n+1)$ matrix which has $1$'s above the diagonal and is 0 everywhere else. Then
\begin{equation*}
\overline{W(T)}=\mathrm{co}\bigl( W\left( B_n+J_n\right)\ \cup\ W\left(B_n-J_n\right) \bigr).
\end{equation*}
\end{conjecture}

The above conjecture is true for $n=1$, as shown in Theorem~\ref{spcase}, see Figure~\ref{n=1}. Cases $n=2$ and $n=3$ have been verified through computer simulations, see Figures \ref{n=2} and \ref{n=3}. However, we are unable to provide a proof yet.

We conclude with an additional observation regarding the symmetry of the set $\overline{W(T)}$. First we prove the following.

\begin{proposition}
Let $J_n$ denote the $(n+1)\times (n+1)$ matrix with value $1$ at the positions $(1,1)$ and $(n+1,n+1)$ and zero everywhere else, and let $B_n$ be the $(n+1)\times (n+1)$ matrix which has $1$'s above the diagonal and is 0 everywhere else. Then
\begin{equation*}
    W\left(B_n+J_n\right)=-W(B_n-J_n).
\end{equation*}
\end{proposition}
\begin{proof}
Given $\vec{x}=(x_1,x_2,\ldots,x_{n+1})$ a vector in $\C^{n+1}$ with $\|x\|=1$, let us denote by 
\[
\vec{y}=(x_1,-x_2,x_3,\dots, (-1)^{n-1}x_{n}, (-1)^{n}x_{n+1})
\]
the vector obtained from $\vec{x}$ by alternating a minus sign in its components. Then $\vec{y}$ also has norm one. Moreover
\begin{align*}
     \inpr{\left(B_n-J_n\right)\vec{y},\vec{y}} ={} &  (-x_1-x_2)\overline{x_1}+x_3(-\overline{x_2})-x_4\overline{x_3}+\cdots \\
     & + (-1)^n x_{n+1} (-1)^{n-1}\overline{x_n} -(-1)^{n}x_{n+1}(-1)^{n}\overline{x_{n+1}}\\
    ={} & -|x_1|^2-x_2\overline{x_1}-x_3\overline{x_2}-x_4\overline{x_3}-\dots-x_{n+1}\overline{x_n}-|x_{n+1}|^2\\
    ={} & -\inpr{\left(B_n+J_n\right)\vec{x},\vec{x}}.
\end{align*}
The statement then follows.
\end{proof}
Hence, the closure of the numerical range of $T$ in Conjecture~\ref{conjecture}, can be further reduced to depend on the numerical range of  a single $(n+1)\times(n+1)$ matrix.



\begin{thebibliography}{10}

\bibitem{BebianoEtAl}
N.~Bebiano, J.~da~Provid\^encia, and A.~Nata.
\newblock The numerical range of banded periodic {T}oeplitz operators.
\newblock {\em J.~Math.~Anal.~Appl.}, { 398}:189–--197, 2013.

\bibitem{BS}
N.~Bebiano and Spitkovsky I.
\newblock Numerical ranges of {T}oeplitz operators with matrix symbols.
\newblock {\em Linear Algebra Appl.}, 436:1721--1726, 2012.

\bibitem{CWChLi10}
S.N. Chandler-Wilde, R.~Chonchaiya, and M.~Lindner.
\newblock Eigenvalue problem meets {S}ierpinski triangle: computing the
  spectrum of a non-self-adjoint random operator.
\newblock {\em Oper.~Matrices}, 5:633--648, 2011.

\bibitem{CWChLi13}
S.N. Chandler-Wilde, R.~Chonchaiya, and M.~Lindner.
\newblock On the spectra and pseudospectra of a class of non-self-adjoint
  random matrices and operators.
\newblock {\em Oper.~Matrices}, 7:739--775, 2013.

\bibitem{ChD}
S.N. Chandler-Wilde and E.B. Davies.
\newblock Spectrum of a {F}einberg-{Z}ee random hopping matrix.
\newblock {\em J.~Spectr.~Theory}, 2:147--179, 2012.

\bibitem{CS}
R.T. Chien and I.M. Spitkovsky.
\newblock On the numerical ranges of some tridiagonal matrices.
\newblock {\em Linear Algebra Appl.}, 470:228--240, 2015.

\bibitem{Feinberg}
J.~Feinberg and A.~Zee.
\newblock Spectral curves of non-hermitean {H}amiltonians.
\newblock {\em Nuclear Phys.~B}, 552:599--623, 1999.

\bibitem{GusRao}
K.~E. Gustafson and D.~K.~M. Rao.
\newblock {\em Numerical range}.
\newblock Universitext. Springer-Verlag, New York, 1997.

\bibitem{HaggerJFA}
R.~Hagger.
\newblock The eigenvalues of tridiagonal sign matrices are dense in the spectra
  of periodic tridiagonal sign operators.
\newblock {\em J.~Funct.~Anal.}, 269:1563--1570, 2015.

\bibitem{Hagger}
R.~Hagger.
\newblock On the spectrum and numerical range of tridiagonal random operators.
\newblock {\em J. Spectr. Theory}, 6:215–266, 2016.

\bibitem{HI-O2016}
C.~Hern\'andez-Becerra and B.~A. Itz\'a-Ortiz.
\newblock A class of tridiagonal operators associated to some subshifts.
\newblock {\em Open Math.}, 14:2391--5455, 2016.

\bibitem{HoJo}
R.~A. Horn and C.~R. Johnson.
\newblock {\em Topics in matrix analysis}.
\newblock Cambridge University Press, Cambridge, 1994.

\bibitem{Kato}
T.~Kato.
\newblock {\em Perturbation theory for linear operators}.
\newblock Die Grundlehren der mathematischen Wissenschaften, Band 132.
  Springer-Verlag New York, Inc., New York, 1966.

\end{thebibliography}
\bibliographystyle{plain}

\end{document}